\documentclass[12pt]{amsart}
\usepackage[english]{babel}
\usepackage{graphicx}
\usepackage{amsfonts}
\usepackage{amsmath}
\usepackage{amssymb}
\usepackage{amscd}
\usepackage[all,cmtip,matrix, arrow]{xy}
\usepackage{amsthm}
\def\Aut {\operatorname{Aut}}
\newcommand{\Sym}{\mathfrak{S}}
\newcommand{\Alt}{\mathfrak{A}}
\newcommand{\Cyc}{\mathfrak{C}}
\newcommand{\Dih}{\mathfrak{D}}

\newtheorem{llemma}{Lemma}[section]
\newtheorem{theorem}[llemma]{Theorem}
\newtheorem{pproposition}[llemma]{Proposition}

\newtheorem{ccorollary}[llemma]{Corollary}

\theoremstyle{definition}
\newtheorem{definition}[llemma]{Definition}
\newtheorem{assumption}[llemma]{Assumption}
\newtheorem{remark}[llemma]{Remark}
\begin{document}
\title{Automorphisms of singular three-dimensional cubic hypersurfaces}
\author{Artem Avilov}
\thanks{This work is supported by the Russian Science Foundation under grant ¹ 18-11-00121}
\address{Laboratory of Algebraic Geometry, National Research University "Higher School of Economics",
6 Usacheva str., Moscow, 119048, RUSSIA
}
\email{v07ulias@gmail.com}

\maketitle{}
\begin{abstract} In this paper we classify three-dimensional singular cubic hypersurfaces with an action of a finite group $G$, which are not $G$-rational and have no birational structure of $G$-Mori fiber space with the base of positive dimension. Also we prove the $\Alt_{5}$-birational superrigidity of the Segre cubic.
\end{abstract}
\section{Introduction}

In this paper we work over an angebraically closed field $\operatorname{k}$ of characteristic 0. Recall that a \emph{$G$-variety} is a pair $(X, \rho)$, where $X$ is an algebraic variety and $\rho: G\to \Aut(X)$ is an injective homomorphism of groups. We say that $G$-variety $X$ has \emph{$G\mathbb{Q}$-factorial singularities} if every $G$-invariant Weil divisor of $X$ is $\mathbb{Q}$-Cartier.

Let $X$ be a $G$-variety with at most $G\mathbb{Q}$-factorial terminal singularities and $\pi:X\to Y$ be a $G$-equivariant morphism. We call $\pi$ a \emph{$G$-Mori fibration} if $\pi_{*}\mathcal{O}_{X}=\mathcal{O}_{Y}$, $\dim X>\dim Y$, the relative invariant Picard number $\rho^{G}(X/Y)$ is equal to $1$ (in this case we say that $G$ is \emph{minimal}) and the anticanonical class $-K_{X}$ is $\pi$-ample. If $Y$ is a point then $X$ is a \emph{$G\mathbb{Q}$-Fano variety}. If in addition the anticanonical class is a Cartier divisor then $X$ is a \emph{$G$-Fano variety}.

Let $X$ be arbitrary normal projective $G$-variety of dimension 3. Resolving the singularities of $X$ and applying the $G$-equivariant minimal model program we reduce $X$ either to a $G$-variety with nef anticanonical class, or to a $G$-Mori fibration (see e.g.~\cite[\S 3]{23}). So such fibrations (and $G\mathbb{Q}$-Fano varieties in particular) form a very important class in the birational classification. In this paper we consider a certain class of $G$-Fano threefolds.

\begin{definition} A projective $n$-dimensional variety $X$ is a \emph{del Pezzo variety} if it has at most terminal Gorenstein singularities and the anticanonical class $-K_{X}$ is ample and divisible by $n-1$ in the Picard group $\operatorname{Pic}(X)$. If a $G$-Fano variety $X$ is a del Pezzo variety, then we say that $X$ is a \emph{$G$-del Pezzo} variety.
\end{definition}

Del Pezzo varieties of arbitrary dimension were classified by T. Fujita (\cite{24},~\cite{25},~\cite{26}). $G\mathbb{Q}$-factorial $G$-minimal three-dimensional $G$-del Pezzo varieties were partially classified by Yu. Prokhorov in~\cite{2}. The main invariant of a del Pezzo threefold $X$ is the \emph{degree} $d=(-\frac{1}{2}K_{X})^{3}$, it is an integer in the interval from 1 to 8. In this paper we consider the case $d=3$. In this case $X$ is a cubic hypersurface in $\mathbb{P}^{4}$. If $d=8$ then $X$ is a projective space. In this case equivariant birational geometry were studied by I. Cheltsov and C. Shramov in the paper~\cite{57}. The case $d=4$ was considered in the author's previous work~\cite{6}. If $d>4$ then $X$ is smooth (cf.~\cite{2}) while smooth del Pezzo threefolds and their automorphism groups are known well. For other types of $G$-Fano threefolds there are only some partial results: see for example~\cite{13},~\cite{17}.

Classification of finite subgroups of the Cremona group $\operatorname{Cr}_{3}(\operatorname{k})$ is one of the motivations of this paper. The Cremona group $\operatorname{Cr}_{n}(\operatorname{k})$ is the group of birational automorphisms of the projective space $\mathbb{P}^{n}_{\operatorname{k}}$. Finite subgroups of $\operatorname{Cr}_{2}(\operatorname{k})$ were completely classified by I. Dolgachev and V. Iskovskikh in~\cite{1}. The core of their method is the following. Let $G$ be a finite subgroup of $\operatorname{Cr}_{2}(\operatorname{k})$. The action of $G$ can be regularized in the following sence: there exists a smooth projective $G$-variety $Z$ and an equivariant birational morphism $Z\to \mathbb{P}^{2}$. Then we apply the equivariant minimal model program to $Z$ and obtain a $G$-Mori fibration which is either a $G$-conic bundle over $\mathbb{P}^{1}$ (which is a blowing up of a Hirzebruch surface at some points), or a $G$-minimal del Pezzo surface. Dolgachev and Iskovskikh classified all minimal subgroups in automorphism groups of del Pezzo surfaces and conic bundles and so they obtained the full list of finite subgroups of $\operatorname{Cr}_{2}(\operatorname{k})$. But quite often two subgroups from such list are conjugate in $\operatorname{Cr}_{2}(\operatorname{k})$, so it is natural to identify them. One can see that $G$-varieties $Z_{1}$ and $Z_{2}$ give us conjugate subgroups if and only if there exists a $G$-equivariant birational map $Z_{1}\dasharrow Z_{2}$. So we need to classify all rational $G$-Mori fibrations and birational maps between them as well.

Following this program in the three-dimemsional case one can reduce the question of classification of all finite subgroups in $\operatorname{Cr}_{3}(\operatorname{k})$ to the question of classification of all rational $G\mathbb{Q}$-Mori fibrations and birational equivariant maps between them. Such program was realized in some particular cases: simple non-abelian groups which can be embedded into $\operatorname{Cr}_{3}(\mathbb{C})$ (see~\cite{10}, see also~\cite{4},~\cite{14},~\cite{29},~\cite{30}) and $p$-elementary subgroups of $\operatorname{Cr}_{3}(\mathbb{C})$ (see~\cite{11},~\cite{7}).

For applications to Cremona groups we are mostly interested in classification of \emph{rational} del Pezzo varieties, so we assume that $X$ is singular (every smooth cubic threefold is not rational due to the classical result of Clemens and Griffiths~\cite{8}). Singular cubic threefolds whose singularities are nodes were classified by H. Finkelnberg and J. Werner in~\cite{9}. They have described the mutual arrangement of singular points and planes, divisorial class groups and small resolutions of such varieties. Following their paper~\cite{9}, we will use the notation J1--J15 for nodal cubic threefolds everywhere below.

In this paper we are interested in the following problem: classify rational $G$-del Pezzo threefolds of degree 3 that have no $G$-equivariant birational map to a ``more simple'' $G$-Fano threefold (for example $\mathbb{P}^{3}$ or a quadric in $\mathbb{P}^{4}$) or to a $G$-Mori fibration with the base of positive dimension. In this paper we give a partial answer for this question.

The main result of this paper is the following theorem:

\begin{theorem}\label{th1} Let $X=X_{3}\subset \mathbb{P}^{4}$ be a cubic hypersurface and $G$ be a finite subgroup of $\Aut(X)$ such that $X$ is a rational $G$-Fano variety. Suppose that $G$ is neither linearizable nor of fiber type and $X$ is not $G$-birationally equivalent to a quadric. Then there is only the following possibilities for $X$ and $G$:
\begin{enumerate}
\item[1.]
 $X=\left\{\sum\limits_{i=0}^{5}x_{i}=\sum\limits_{i=0}^{5}x_{i}^{3}=0\right\}\subset\mathbb{P}^{5}$, i.e. $X$ is the Segre cubic, and $G$ is $\Aut(X)=\Sym_{6}$, $\Alt_{6}$, $\Sym_{5}$ or $\Alt_{5}$, two last subgroups are standard;
\item[2.]
 $X=\{x_{0}x_{1}x_{2}-x_{3}x_{4}x_{5}=\sum\limits_{i=0}^{5}x_{i}=0\}\subset\mathbb{P}^{5}$ and $G$ is $\Aut(X)=\Sym_{3}^{2}\rtimes\Cyc_{2}$, $\Sym_{3}^{2}$ (which acts transitively on the set of singularities) or $\Cyc_{3}^{2}\rtimes\Cyc_{4}$.
\end{enumerate}
Here by $\Cyc_{n}$ we denote a cyclic group of order $n$, by $\Sym_{n}$ we denote a symmetric group of degree $n$ and by $\Alt_{n}$ we denote an alternating group of degree $n$. Moreover, all $G$-varieties of the first type are $G$-birationally superrigid.
\end{theorem}

The paper organized as follows. In the second section we reduce the first part of the theorem to five cases: J15 (the Segre cubic), J14, J11, J9 and the cases of five singularities (see~\cite{9} for explanation of J1--J15). In the third section we explain the notion of birational rigidity and remind some well-known theorems which helps us to exclude potential non-canonical centers. Then in the fourth section we study the case of the Segre cubic and check its birational rigidity with respect to a non-standard subgroup $\Alt_{5}\subset\Aut(X)$. In the fifth section we consider the case case J14. In the last section we exclude the last three cases which concludes the proof of the theorem~\ref{th1}.

In this paper we use the following notation:
\begin{itemize}
\item
$\Cyc_{n}$ is a cyclic group of order $n$;
\item
$\Dih_{2n}$ is a dihedral group of order $2n$;
\item
$\Sym_{n}$ is a symmetric group of degree $n$;
\item
$\Alt_{n}$ is an alternating group of degree $n$;
\item
$G^{n}$ is the direct sum of $n$ copies of the group $G$;
\item
$(a_{1}, a_{2}, ..., a_{n_{1}})(a_{n_{1}+1}, ..., a_{n_{2}})...(a_{n_{m-1}+1}, ..., a_{n_{m}})$ is a cyclic decomposition of an element $\sigma\in\Sym_{N}$;
\item
$\Sym_{5}\subset\Sym_{6}$ is a \emph{standard} subgroup iff it is the stabilizer of some element in $\{1, ..., 6\}$ for the natural action of $\Sym_{6}$, and \emph{non-standard} otherwise. We use the same notation for $\Alt_{5}\subset\Sym_{6}$;
\item
$\delta^{i}_{j}$ is the Kronecker symbol: $\delta^{i}_{j}=1$ if $i=j$ and $\delta^{i}_{j}=0$ otherwise.
\end{itemize}

The author is a Young Russian Mathematics award winner and would like to thank its sponsors and jury. The author also would like to thank Ivan Cheltsov, Yuri Prokhorov, Constantin Shramov and Andrey Trepalin for useful discussions and comments.

\section{Singularities of three-dimensional cubic hypersurfaces}

\begin{definition} Let $X$ be a Fano threefold with at worst terminal singularities. Let $G\subset\Aut(X)$ be a finite subgroup. We say that $G$ is \emph{minimal} if $X$ has $G\mathbb{Q}$-factorial singularities and $\rho^{G}(X)=1$. In our case it is equivalent to $\operatorname{rk}\operatorname{Cl}(X)^{G}=1$.
\end{definition}
\begin{definition} We call a group $G\subset\Aut (X)$ \emph{linearizable} (resp., \emph{of fiber type}) if there exists a $G$-equivariant birational map $X\dasharrow \mathbb{P}^{3}$ (resp., $X\dasharrow X'$ where $X'$ has a structure of a $G$-Mori fibration with the base of positive dimension).
\end{definition}
\begin{assumption}\label{as} Throughtout this paper $X$ is a singular cubic hypersurface in $\mathbb{P}^{4}$ with only terminal singularities and $G$ is a finite subgroup of $\Aut(X)$ such that $G$ is minimal and is neither linearizable nor of fiber type (so $\Aut(X)$ also has such properties).
\end{assumption}

The inclusion $X\to\mathbb{P}^{4}$ is given by the linear system $|-\frac{1}{2}K_{X}|$ hence the action of $G$ on $X$ is induced from a linear action of $G$ on $\mathbb{P}^{4}$.

The following lemma is well-known.
\begin{llemma}\label{le88} Let $Y$ be a three-dimensional $G$-variety and let $Z$ be a $G$-variety of dimension $1$ or $2$. Let $Y\dasharrow Z$ be a $G$-equivariant dominant rational map whose general fiber is either a rational curve or a rational surface respectively. Then there exists the following commutative diagram
\begin{equation}
\label{diag}\xymatrix{
Y \ar@{-->}[r]\ar@{-->}[d] & Y' \ar[d]\\
Z & Z' \ar@{-->}[l]
}
\end{equation}
where $Y'\to Z'$ is a $G$-Mori fiber space and both horizontal maps are birational and $G$-equivariant.
\end{llemma}
\begin{proof} Note that one can equivariantly compactify a $G$-variety $S$: we can compactify the quotient $S/G$ and take its normalization in the function field $\operatorname{k}(S)$. Therefore  we may assume that $Y$ and $Z$ are projective. Then we can equivariantly resolve the singularities (see~\cite{20}). So we may assume that $Y\dasharrow Z$ is a $G$-equivariant morphism between smooth projective varieties. Then we apply the $G$-equivariant relative minimal model program to $Y$ over $Z$. Since fibers of the map $Y\to~Z$ are rationally connected, at the end we obtain a required $G$-Mori fibration $Y'\to Z'$. This gives us the required commutative diagram.
\end{proof}

\begin{remark} In the case where $Y\dasharrow Z$ is of relative dimension one the commutative diagram~\eqref{diag} can be taken so that $Y'$ and $Z'$ are smooth varieties (see~\cite{12} for details).
\end{remark}

The following easy lemma is very important:

\begin{llemma}\label{le1}~
\begin{enumerate}
\item
The variety $X$ has no $G$-fixed singular points and no $G$-invariant lines;
\item there are no $G$-invariant planes in $\mathbb{P}^{4}$.
\end{enumerate}
\end{llemma}
\begin{proof} If $X$ contains a fixed  singular point then the projection from this point is a $G$-equivariant birational map to $\mathbb{P}^{3}$. Therefore $G$ is linearizable in this case. If $X$ contains a $G$-invariant line then the projection from this line gives us a $G$-equivariant rational curve fibration, and in the case of $G$-invariant plane we get a $G$-fibration by quadric or cubic surfaces in $\mathbb{P}^{3}$ (the first case occurs exactly when the plane lies on~$X$). Then we apply Lemma~\ref{le88}. Therefore $G$ is of fiber type in these cases.
\end{proof}

\begin{llemma}\label{le3} The $G$-orbit of a singular point cannot consist of $4$ elements.
\end{llemma}
\begin{proof}
Assume the opposite: there is a $G$-orbit in $\operatorname{Sing}{X}$ of length 4. If this points are coplanar then $X$ is of fiber type by Lemma~\ref{le1}. So we may assume that they are in general position. Let $S$ be a hyperplane section of $X$ passing through all the singular points from this orbit. It is a singular cubic surface (maybe reducible) with at least four singular points in general position. Due to~\cite{16} such a surface must be either the special cubic surface with four nodes or reducible. In the second case $S$ must be either a union of three planes with unique common point or a union of a quadric cone whose vertice is a singular point of $X$ and a plane, otherwise singularities of $X$ cannot be in general position. In both cases $\Aut(X)$ cannot act transitively on $\operatorname{Sing}(X)$: if $S$ is a union of three planes either the common point of planes is a distinguished singularity of $X$ or there is a distinguished plane that contains exactly two singularities of $X$; if $S$ is a union of a quadric cone and a plane then the vertice of the cone is a distinguished singularity of $X$. One can easily check that the cubic surface with four singularities of type $A_{1}$ in suitable coordinate system has the equation
$$F(x_{0}, x_{1}, x_{2}, x_{3})=x_{0}x_{1}x_{2}+x_{0}x_{1}x_{3}+x_{0}x_{2}x_{3}+x_{1}x_{2}x_{3}=0.$$

Let us consider a canonical map $\pi:\Aut(S)\to \Sym_{4},$ where $\Sym_{4}$ is a group of permutations of singular points. Every element of $\operatorname{PGL}_{4}(\operatorname{k})$ preserving all singular points is a diagonal map. One can easily check that diagonal map preserves $S$ if and only if it is trivial, so the map $\pi$ is injective. On the other hand, the group $\Sym_{4}$ acts on $S$ by permutation of coordinates, so $\Aut(X)\simeq\Sym_{4}$. But the plane $x_{0}+x_{1}+x_{2}+x_{3}=0$ is $\Aut(S)$-invariant and, consequently, $G$-invariant. Thus the group $G$ is of fiber type by Lemma~\ref{le1}.
\end{proof}

\begin{ccorollary}\label{sl1}
The $G$-orbit of a singular point of $X$ has length at least~$5$.
\end{ccorollary}

\begin{definition} We say that points  $p_{1},..., p_{k}\in \mathbb{P}^{n}$ are \emph{in general position} if no $d$ of them lie in a subspace $\mathbb{P}^{d-2}\subset \mathbb{P}^{n}$ for every $d\leq n+1$.
\end{definition}
The following facts about singular points of cubic threefolds are well-known.

\begin{llemma}\label{le2} Let $Y\subset\mathbb{P}^{4}$ be an arbitrary cubic hypersurface with isolated singularities. Then
\begin{enumerate}
\item no three singular points lie on one line;
\item if four singular points lie on a plane then this plane is contained in $Y$ and there are no other singular points on it;
\item if no four singular points lie on a plane then all singular points of $Y$ are in general position.
\end{enumerate}
\end{llemma}
\begin{proof} The first statement can be easily deduced from the fact that the singular set of $X$ is an intersection of quadrics.

Assume that four singular points of $Y$ lie on a plane $P$. Suppose that $P$ is not contained in $Y$. Then $Y\cap P$ is a cubic curve with four singular points such that no three of them lie on one line. It is impossible, so $P$ lies on $Y$. The second part of (ii) also follows from the fact that the singular set of $X$ is an intersection of quadrics.

Let $H$ be a hyperplane such that it contains at least five singular points of $Y$ and no four of them lie on a plane. Consider the intersection $Z=Y\cap H$. It is a cubic surface with at least five singular points. Due to~\cite{16} such a surface must be reducible. If $Z$ is a union of a quadric and a plane, then at least four singular points lie on this plane. If $Z$ is a union of three different planes then all singular points lie on three lines  (each line is the intersection of two planes). Thus in this case we again see, that one of three planes contains four singular points of $Y$. If $Z$ contains a double or triple plane, then such plane is the singular set of $Z$, so it contains five singular points of $Y$. This contradiction proves (iii).
\end{proof}

\begin{pproposition}\label{ut1} Assume that all singularities of $X$ are nodes. Then there is one of the following possibilities:

\begin{center}
{\rm
\begin{tabular}{|p{3,5cm}|p{1.1cm}|p{1.1cm}|p{1.1cm}|p{1.1cm}|p{1.1cm}|}

\hline
type of $X$ in terms of~\cite{9} & J5 & J9 & J11 & J14 &J15\\
\hline
$s(X)$ & 5 & 6 & 6 & 9 & 10\\
\hline
$p(X)$ & 0 & 0 & 3 & 9 & 15\\
\hline
$r(X)$ & 1 & 2 & 3 & 5 & 6\\
\hline
\end{tabular}
}
\end{center}

where $s(X)$ is the number of nodes, $p(X)$ is the number of planes on $X$ and $r(X)$ is the rank of $\operatorname{Cl}(X)$.
\end{pproposition}
\begin{proof} Cases J1--J4 of~\cite{9} are impossible because $X$ has at least five singular points by Corollary~\ref{sl1}. If $X$ is of type J6, J7 or J8, then $X$ contains exactly one plane, which is $\Aut(X)$-invariant. If $X$ is of type J10 or J12, then there is a distinguished $\Aut(X)$-invariant singular point of $X$ ($p_{7}$ more presicely). If $X$ is of type J13, then there is a distinguished quadruple of singular points $p_{5}, p_{6}, p_{7}, p_{8}$ which lie on one $\Aut(X)$-invariant plane. Hence in all these cases we can apply Lemma~\ref{le1} to obtain a contradiction. Numbers $s(X)$, $p(X)$ and $r(X)$ can be found in~\cite{9}.
\end{proof}

Now we consider the case where not all the singularities of $X$ are nodes.
\begin{pproposition}\label{ut2} Assume that $X$ contains a singularity which is not a node. Then the variety $X$ has exactly five singularities of type $cA_{1}$ or $cA_{2}$ in general position and no other singularities.
\end{pproposition}
\begin{proof} We have the following formula for the degree of the dual variety: $$\operatorname{deg}X^{\vee}=3\cdot 2^{3}-\sum\limits_{p\in\operatorname{Sing}(X)} m(p),$$ where $m(p)=\mu(p)+\mu'(p)$ is the sum of Milnor numbers of a singularity $(X, p)$ and its hyperplane section (see~\cite{15} for details). Obviously, we have $\operatorname{deg}X^{\vee}\geq 3$, hence $$\sum\limits_{p\in\operatorname{Sing}(X)} m(p)\leq 21.$$ Note that three-dimensional terminal hypersurface singularities are exactly cDV points (see~\cite[Corollary 5.38]{28}). If $p$ is a cDV singularity which is not of type $cA_{1}$, then $\mu(p)\geq 2$ and $\mu'(p)\geq 2$ (one can see it easily from the definition of Milnor number), and equalities hold exactly for $cA_{2}$ singularities. Thus singularities are of type $cA_{2}$ and we have five of them. The number of remaining singularities cannot be greater than $\frac{21-5\cdot 4}{2}=\frac{1}{2}$, so there are no other singularities on $X$.

Assume that all singularities of $X$ are $cA_{1}$ points. In some analytic neighbourhood they can be given by an equation $$x^{2}+y^{2}+z^{2}+t^{n}=0,\ n>1.$$ Hence they are \emph{rs-nondegenerate} (see~\cite[Definition 10.1, Proposition 10.3]{2}). On the other hand, $r(X)=\operatorname{rk}\operatorname{Cl}(X)\leq 3$ by~\cite[Theorems 1.7, 7.1 and 8.1]{2}. Thus we can apply the following formula (see~\cite[Proposition 10.6]{2}):
\begin{equation}\label{eq9}
\sum\limits_{p\in\operatorname{Sing}(X)}\lambda (X, p)\leq r(X)-\rho(X)+h^{1, 2}(X')-h^{1, 2}(\widehat{X})\leq 7,
\end{equation}
where $X'$ is a general smooth cubic, $\rho(X)$ is the Picard number of $X$, variety $\widehat{X}$ is a standard resolution of $X$ (see~\cite[Definition 10.1]{2}) and $\lambda(X, p)$ is the number of exceptional divisors of $\widehat{X}\to X$ over $p$. It follows from~\eqref{eq9} and Corollary~\ref{sl1} that all singular points lie in one $\Aut(X)$-orbit and $\lambda(X, p)=1$ for every singular point $p$. By the direct computations one can easily see that $\lambda(X, p)=1$ iff $n\leq3$. Hence $n=3$. Assume that the number of singular points is greater than 5. Then $r(X)\geq 2$ by the formula~\eqref{eq9} and $X$ is not $\mathbb{Q}$-factorial. But the singularity $x^{2}+y^{2}+z^{2}+t^{3}=0$ is $\mathbb{Q}$-factorial (see~\cite[Corollary 1.16]{31}), a contradiction. Thus the number of $cA_{1}$ points which are not nodes cannot be greater than 5.

As an easy consequence from the Lemma~\ref{le2}, all singular points are in general position.
\end{proof}

\begin{remark} Later we will show (see \S6) that two cases from the statement of the Proposition~\ref{ut2} cannot occur.
\end{remark}

\section{Birational rigidity and singularities of linear systems}

Let $X$ be a variety with at most terminal singularities and let $G$ be a finite group acting on $X$ such that $X$ is a $G\mathbb{Q}$-Fano variety.

\begin{definition} $G$-Fano variety $X$ is \emph{$G$-birationally rigid} if for every $G$-Mori fibration $X'\to Y$ such that $X$ and $X'$ are $G$-birationally equivalent $X\simeq X'$. If moreover $\operatorname{Bir}(X)=\Aut(X)$ then $X$ is \emph{$G$-birationally superrigid}.
\end{definition}

For checking $G$-birational rigidity of variety $X$ we need to consider all $G$-invariant linear subsystems $\mathcal{M}\subset |-\mu K_{X}|$ for rational $\mu$ without fixed components. For every $\mathcal{M}$ we need to find all centers of non-canonical singularities and Sarkisov links for them. If there is no non-canonical centers then $X$ is $G$-birationally superrigid. If all Sarkisov links are birational automorphisms of $X$ then $X$ is birationally rigid (see~\cite{56} for details).

For excluding possible non-canonical centers we need the following theorems:

\begin{theorem}\label{th11}\textnormal{(see, e.g.~\cite[Lemma 1.10]{31})} In our notation let a smooth point $p\in X$ be a center of a non-canonical singularity of the pair $(X, \frac{1}{\mu}\mathcal{M}).$ Let cycle $Z=M_{1}\cdot M_{2}$ be an intersection of two general members of the linear system $\mathcal{M}$. Then $\operatorname{mult}_{p}Z>4\mu^{2}.$
\end{theorem}

\begin{theorem}\label{th10}\textnormal{(see, e.g.~\cite[Theorem 1.7.20]{32})} Let $p\in X$ be an ordinary double point. Let $D$ be an effective $\mathbb{Q}$-divisor on $X$ such that the pair $(X, D)$ is not canonical at $p$. Then $\operatorname{mult}_{p}D>1$.
\end{theorem}

\begin{theorem}\label{th12}\textnormal{(see, e.g.~\cite[Exercise 6.18]{33})} In our notation let $C$ be an irreducible curve that is a center of a non-canonical singularity of the pair $(X, \frac{1}{\mu}\mathcal{M}).$ Then the multiplicity of $\mathcal{M}$ along $C$ is greater than $\mu$.
\end{theorem}
\section{Segre cubic}
In this section we consider the case where $X$ is a cubic hypersurface in $\mathbb{P}^{4}$ of type J15 satisfying Assumption~\ref{as} and $G$ is the corresponding minimal subgroup of $\Aut(X)$. Such a variety is unique up to isomorphism. It is called \emph{Segre cubic} and can be explicitly given by the following system of equations: $$\sum\limits_{i=0}^{5} x_{i}=\sum\limits_{i=0}^{5} x_{i}^{3}=0$$ in $\mathbb{P}^{5}$. This variety has many interesting properties (see e.g.~\cite[\S 9.4.4]{21}):
\begin{itemize}\item it is a unique cubic threefold with ten isolated singularities.
\item the automorphism group of $X$ is isomorphic to $\Sym_{6}$ and acts on $X$ via permutation of coordinates.
\item the variety $X$ contains exactly 15 planes forming an $\Aut(X)$-orbit and one of them is given by equations $$x_{1}+x_{2}=x_{3}+x_{4}=x_{5}+x_{6}=0.$$
\item singular points of $X$ form an $\Aut(X)$-orbit, the point $(1:1:1:-1:-1:-1)$ is one of them.
\end{itemize}

\begin{pproposition}The group $G$ coincides with one of the following subgroups of $\Aut(X)\simeq \Sym_{6}$:
$$\Alt_{5},\ \Sym_{5},\ \Alt_{6},\ \Sym_{6},$$ where $\Sym_{5}$ and $\Alt_{5}$ are standard subgroups.
\end{pproposition}
\begin{proof} One can easily prove this using the list of all subgroups of $\Sym_{6}$ (for example, one can use GAP~\cite{27} to construct such a list) and following simple facts.
\begin{itemize}
\item
The group $G$ with the natural action on the set of coordinates cannot have an invariant pair of coordinates: otherwise there is a $G$-invariant plane in $\mathbb{P}^{4}$.
\item
The group $G$ is not a subgroup of $\Sym_{3}^{2}\rtimes \Cyc_{2}\subset\Aut(X)$ since such group is a stabilizer of some singular point of $X$ (see~\cite[p. 252]{3}).
\item
Every subgroup $\Sym_{5}\subset\Aut(X)$ acting transitively on the set of coordinates has an orbit of length 5 in the set of planes in $X$ (see~\cite[Lemma 3.8]{3}). The sum of all planes in such orbit cannot be proportional to the canonical class of $X$. Hence such subgroups $\Sym_{5}$ are not minimal. As a consequence, $G$ is not a subgroup of such group.
\item
Every subgroup $H\simeq \Sym_{4}\times \Cyc_{2}$ acting transitively on the set of coordinates is the stabilizer of some plane on $X$ since all such subgroups are conjugate and the stabilizer of a plane has the same property. As a consequence, $G$ is not a subgroup of such group.
\end{itemize}
Using these facts we exclude all the possibilities for $G$ except four classes from the assertion. On the other hand all these four classes of subgroups act transitively on the set of planes, so they are minimal.
\end{proof}

\begin{remark} I. Cheltsov and C. Shramov~\cite{4} proved that the Segre cubic with the action of $G=\Alt_{6}$ has no equivariant birational transformations to another $G$-Mori fibration (in other words, the Segre cubic is $G$-birationally rigid). So the $G$-birational rigidity of Segre cubic with the action of $\Alt_{5}$ and $\Sym_{5}$ is the only open problem related to the $G$-rigidity property of the Segre cubic. There are some other examples of $\Alt_{5}$-birationally rigid Fano threefolds (see \cite{4}, \cite{29}).
\end{remark}

Now we will prove that $X$ is $G$-birationally superrigid if $G\simeq \Alt_{5}$ is a standard subgroup. Without loss of generality we may assume that $G$ fixes the coordinate $x_{0}$. Let $H_{0}=\{x_{0}=0\}$ be a hyperplane and $S=X\cap H_{0}$. It is a smooth cubic surface named Clebsch cubic.
\begin{llemma}\label{le4} Assume that $O$ is an orbit of some point of $X$. Then one of the following possibilities holds:
\begin{itemize}
\item $O$ consists of 5 elements. In this case points of $O$ are in general position and a line through any two of them does not lie on $X$;
\item $O$ contains three colinear points such that the line passing through them does not lie on $X$;
\item $O$ contains six points such that any three of them are not colinear and any five of them are not coplanar.
\end{itemize}
\end{llemma}
\begin{proof} Assume that the length of $O$ is less than 10. One can easily check that $O$ is the orbit of the point $(-1:1:0:0:0:0)$ and satisfies the first possibility. From now $O$ contains at least 10 elements.

Assume that $O\subset X\setminus S$. Since $O$ generates $\mathbb{P}^{4}$ ($H_{0}$ is the only $G$-invariant subspace), we can choose subset $P=\{p_1, p_2, p_3, p_4, p_5\}\subset O$ of points in general position. Suppose that any other point of $O$ lies on the line $l_{ij}$ passing through $p_{i}$ and $p_{j}$ for some $i, j$. Then one of two possibilities holds: either there are two points $p_{6}$ and $p_{7}$ in $O\setminus P$ on $l_{ij}$ and $l_{ik}$ respectively for some $i, j, k\in\{1,..., 5\}, j\neq k$, or all points of $O\setminus P$ lies on $l_{ij}$ and $l_{kl}$ where $i, j, k, l$ are distinct. In the second case for every point $p\in O$ all points of $O\setminus {p}$ lie on a hyperplane. Obviously, this is impossible. In the first case the set $(P\cup \{p_{6}, p_{7}\})\setminus\{p_{i}\}$ satisfy the third possibility.

Assume that $O\subset S$.

\begin{llemma}\label{lem15} Let $V_4=\{\sum x_{i}=0\}\subset V_{5}$ be the simplicial $4$-dimensional representation of $\Alt_{5}$. Let $v\in V$ be a nonzero vector such that it is not proportional to a vector in the orbit of $(3\xi+2\xi^{2}+\xi^{3}, 3\xi^{4}+2\xi^{3}+\xi^{2}, 1, 1, 1)$, where $\xi$ is a root of unity of fifth degree. Then there exist a subgroup $F\subset\Alt_{5}, F\simeq\Cyc_{5}$ such that $F$-orbit of $v$ generates $V_{4}$.
\end{llemma}
\begin{proof} It can be done by the direct computations in the following way. If for every $\Cyc_{5}\subset \Alt_{5}$ the $\Cyc_{5}$-orbit of $v$ doesn't generate $V_{4}$ then $v$ lies in a $3$-dimensional subrepresentation of $\Cyc_{5}$. For every subgroup $\Cyc_{5}\subset \Alt_{5}$ we have four $3$-dimensional subrepresentations. Using computer one can easily see, that if we take a $3$-dimensional subrepresentations of $\Cyc_{5}$ for every $\Cyc_{5}\subset \Alt_{5}$ then their intersection is either trivial or $1$-dimensional and generated by a vector of required form.
\end{proof}
\begin{ccorollary} For any point $p\in S$ one can find a subgroup $G\subset\Alt_{5}, G\simeq\Cyc_{5}$ such that $p$ is not a fixed point of $G$ and points in $G$-orbit of $p$ are in general position.
\end{ccorollary}
\begin{proof} One can easily check that point $(3\xi+2\xi^{2}+\xi^{3}:3\xi^{4}+2\xi^{3}+\xi^{2}:1:1:1)$, where $\xi$ is a root of unity of fifth degree, is not a point of $S$, so we can apply the previous lemma.
\end{proof}
By the previous corollary, there is a $\Cyc_{5}$-orbit $\{p_1, p_2, p_3, p_{4}, p_5\}=P\subset O$ such that points of $P$ are in general position. If there is a point $p\in O\setminus P$ that does not lie on the line passing through any two points of $P$ then the third possibility holds. In other case we can assume without loss of generality that $p\in l_{12}$, where $l_{ij}$ is the line passing through $p_{i}$ and $p_{j}$. If $l_{12}\not\subset S$ then the second possibility holds. Otherwise $l_{12}, l_{23}, l_{34}, l_{45}, l_{51}$ are lines on $S$ and $O$ coincides with $\Sigma_{10}$ (unique orbit of 10 elements on $S$) or $\Sigma_{15}$ (unique orbit of 15 elements on $S$) by~\cite[Lemma 6.3.12]{29}. But this is impossible by~\cite[Lemma 6.3.13]{29} since $p_{1}$, $p_{2}$ and $p$ are colinear.
\end{proof}

Let $\mathcal{M}\subset|-\mu K_{X}|$ be an $\Alt_{5}$-invariant linear system without fixed components.
\begin{llemma}\label{le5} Point is not a non-canonical center of the pair $\left(X, \frac{1}{\mu}\mathcal{M}\right)$.
\end{llemma}
\begin{proof} Assume that a singular point $p$ is a non-canonical center. Since the group $\Alt_{5}$ transitively acts on the set $\operatorname{Sing}(X)$, all singular points are non-canonical centers. Let $P\subset X$ be a plane and let $Q\subset P$ be a general conic passing through 4 singular points. Conic $Q$ does not lie on $M_{1}$ since $\mathcal{M}$ has no base components. So by Theorem~\ref{th10}
$$4=Q\cdot\left(\frac{1}{\mu}M_{1}\right)\geq 4\cdot\operatorname{mult}_{p}\mathcal{M}>4.$$
Contradiction. Thus singular points are not non-canonical centers.

Assume that smooth point $p$ is a non-canonical center. Let us apply Lemma~\ref{le4}. Let $M_{1}$ and $M_{2}$ be general elements of $\mathcal{M}$ and let $Z=M_{1}\cdot M_{2}$. If the first possibility holds then the linear system $\mathcal{C}$ of cubic hypersurfaces in $\mathbb{P}^{4}$ with singularities at the orbit of $p$ has no base curves on $X$. Let $Y$ be a general element of $\mathcal{C}$. Then by Theorem~\ref{th11}
$$36\mu^{2}=M_{1}\cdot M_{2}\cdot Y=Z\cdot Y>5\cdot 4\mu^{2}\cdot 2=40\mu^{2}.$$
Assume that the second possibility holds. Let $l\not\subset X$ be a line passing through three point in the orbit of $p$ and let $L$ be a general hyperplane passing through $l$. Then by Theorem~\ref{th11}
$$12\mu^{2}=M_{1}\cdot M_{2}\cdot L=Z\cdot L>3\cdot 4\mu^{2}=12\mu^{2}.$$
Finally, assume that the third possibility holds. Then there exists a linear system of quadrics passing through 6 points in the orbit of $p$ without base curves on $X$. Let $Q$ be a general member of such linear system. Then by Theorem~\ref{th11}
$$24\mu^{2}=M_{1}\cdot M_{2}\cdot Q=Z\cdot L>6\cdot 4\mu^{2}=24\mu^{2}.$$
In any case we obtain a contradiction.
\end{proof}
\begin{llemma}\label{le6} Curve is not a non-canonical center of the pair $\left(X, \frac{1}{\mu}\mathcal{M}\right)$.
\end{llemma}
\begin{proof} Let $C_{1}$ be an irreducible curve which is a non-canonical center and let $\{C_{i}\mid 1\leq i\leq r\}$ be its $\Alt_{5}$-orbit. Let us denote $C=\sum\limits_{i=1}^{r}C_{i}$ and $d=\deg C_{1}$. Let $H$ be a general hyperplane section of $X$. Then by Theorem~\ref{th12}
$$12\mu^{2}=H\cdot M_{1}\cdot M_{2}\geq (\operatorname{mult}_{C}\mathcal{M})^{2}C\cdot H>\mu^{2}rd,$$
so $rd<12$. Since $r$ is an index of some subgroup of $\Alt_{5}$, the number $r$ can be equal to $1, 5, 6$ or $10$.

Assume that $C_{1}\subset S$. We know that $S$ is the Clebsch cubic. By~\cite[Theorem 6.3.18]{29} $\deg C=6$. Since $\mathcal{M}$ has no base components, we obtain
$$6\mu=\mathcal{M}\cdot H_{0}\cdot H\geq(\operatorname{mult}_{C}\mathcal{M})C\cdot H>6\mu,$$
where $H$ is a general hyperplane. Contradiction. Thus we may assume that $C\not\subset S$.

On the one hand $C\cdot H_{0}\leq 11$. On the other hand the orbit of the point $(0:0:0:0:1:-1)$ is the unique $\Alt_{5}$-orbit in $S$ of length less than 12 (see~\cite[Lemma 6.3.12]{29}). So $rd=10$ and $r=1, 5$ or 10. Assume that $r=1$. Then $C$ is an irreducible curve with an action of $\Alt_{5}$. Let $\widetilde{C}$ be its normalization. Note that all points in the intersection of $S$ and $C$ are smooth points of $C$, otherwise the intersection product $C\cdot S$ is greater than $10$. Hence the curve $\widetilde{C}$ with non-trivial action of $\Alt_{5}$ has an $\Alt_{5}$-orbit of length $10$, which contradicts~\cite[Lemma 5.14]{29}. So $r$ cannot be equal to $1$.

Assume that $r=5$. In this case $C_{1}$ is a conic, and $\operatorname{Stab}C_{1}\subset\Alt_{5}$ is isomorphic to $\Alt_{4}$. Without loss of generality we may assume that $\operatorname{Stab}C$ fixes the coordinate $x_{1}$. The plane $P$ which contains $C$ gives us 3-dimensional subrepresentation of $\Alt_{4}$ in our 5-dimensional representation. One can easily see that 5-dimensional representation is a direct sum of two trivial representations and one simplicial 3-dimensional representation, so $P$ is a projectivisation of such simplicial representation. It can be given by the following equations: $x_{0}=x_{1}=\sum\limits_{i=0}^{5}x_{i}$, hence $C\subset S$. Contradiction.

Finally, suppose that $r=10$, in this case $C_{1}$ is a line, $G=\operatorname{Stab}{C_{1}}\simeq\Sym_{3}$ and $C_{1}\cap S$ is a $G$-invariant point. Without loss of generality we may assume that $C_{1}\cap S=(0:0:0:0:1:-1)$. The group $\Sym_{3}$ cannot act faithfully on $\mathbb{P}^{1}$ with a fixed point, so $\Cyc_{3}\subset G$ acts trivially on $C_{1}$ and there is another point $p\in C_{1}$ fixed by $G$. There are three such points on $X\setminus S$: $(1:a:a:a:b:b)$, where $1+3a+2b=1+3a^{3}+2b^{3}=0$. But the line through $(0:0:0:0:1:-1)$ and $(1:a:a:a:b:b)$ do not lie on $X$: for example, the point $(1:a:a:a:b+1:b-1)$ do not a point of $X$.

So, we exclude all possible cases and $C$ is not a non-canonical center.
\end{proof}
\begin{theorem} The Segre cubic with an action of standard subgroup $\Alt_{5}\subset\Sym_{6}$ is $G$-birationally superrigid.
\end{theorem}
\begin{proof} It is a direct corollary of Lemma~\ref{le5} and Lemma~\ref{le6}.
\end{proof}
\section{Cubic hypersurfaces with nine singular points}
In this section we consider the case where  $X$ is a cubic hypersurface in $\mathbb{P}^{4}$ of type J14 satisfying Assumption~\ref{as} and $G$ is the corresponding minimal subgroup of $\Aut(X)$. In this case $X$ is a hyperplane section of the cubic fourfold $Z\subset\mathbb{P}^{5}$, which can be explicitly given by the following equation: $x_{1}x_{2}x_{3}=y_{1}y_{2}y_{3}$ (see~\cite[Prop. 2.2]{5}). One can easily see that the group $\Aut(Z)$ is isomorphic to $(\operatorname{k}^{*})^{4}\rtimes (\Sym_{3}^{2}\rtimes \Cyc_{2})$ where the algebraic torus $(\operatorname{k}^{*})^{4}$ acts on the coordinates diagonally and the subgroup $\Sym_{3}^{2}\rtimes \Cyc_{2}$ acts naturally by permutation of the coordinates (see~\cite[Lemma 1.1]{5}). The singular locus $\operatorname{Sing}(Z)$ consists of nine lines $$l_{ij}=\{x_{k}=y_{l}=0,\ k\neq i, l\neq j\}$$ which intersect $X$ in the singular locus of $X$. Also $Z$ contains nine $3$-spaces $$M_{ij}=\{x_{i}=y_{j}=0\}$$ and the intersections of these subspaces with $X$ are planes on $X$. Note that there is a natural $\Aut(X)$-equivariant bijection $l_{ij}\leftrightarrow M_{ij}$ between lines and $3$-spaces on $Z$.

Suppose that $$a_{1}x_{1}+a_{2}x_{2}+a_{3}x_{3}+b_{1}y_{1}+b_{2}y_{2}+b_{3}y_{3}=0$$ is an equation of the hyperplane which cuts out $X$. Notice that $a_{i}\neq 0$ and $b_{i}\neq 0$ for all $i$. Indeed, otherwise some lines from $\operatorname{Sing}(Z)$ meet $X$ at the same point, that is impossible. Using some diagonal coordinate change one can achieve $a_{1}=a_{2}=a_{3}=A,\ b_{1}=b_{2}=b_{3}=1$ for some number $A$. One can easily see that the hyperplane sections corresponding to $A$ and $\xi A, \xi^{3}=1$ are isomorphic. Note that if $A=-1$ then $(1:1:1:1:1:1)$ is a singular point of $X$ and in fact $X$ is the Segre cubic. Moreover, if $A^{3}\neq -1, 0$ then $X$ contains 9 nodes and no other singularities.

\begin{llemma} Let $Y=Z\cap H$ be a hyperplane section of $Z$ such that $Y$ is a cubic hypersurface in $\mathbb{P}^{4}$ with exactly 9 nodes as singularities. Then every automorphism $\sigma\in \Aut(Y)$ is induced by some automorphism $\bar{\sigma}\in \Aut(Z)$.
\end{llemma}
\begin{proof} In fact, a proof of this proposition is contained implicitly in the proof of~\cite[Prop. 2.2]{5}. For the convenience of the reader we reproduce a complete proof here.

Let $\sigma$ be an automorphism of $Y$. Then there exist a (non-unique) automorphism $\widetilde{\sigma}\in \operatorname{PGL}_{6}(\operatorname{k})$ such that $\widetilde{\sigma}|_{Y}=\sigma$. Denote $Z'=\widetilde{\sigma}(Z)$. Notice that $Z'\cap H=Z\cap H$ where $H$ is a hyperplane which cuts out $Y$. So it is enough to prove that there exist such an automorphism $\tau \in \operatorname{PGL}_{6}(\operatorname{k})$ that $\tau(Z)=Z'$ and $\tau|_{H}=\operatorname{Id}$. Let $B$ be a subgroup of $\operatorname{PGL}_{6}(\operatorname{k})$ that consists of all elements acting trivially on $H$. Obviously, $G$ acts transitively on $\mathbb{P}^{5}\setminus H$.

The variety $Z$ contains 9 three-dimensional projective subspaces $M_{ij}$. So $Z'$ also contains nine $3$-spaces. Denote them by $M_{ij}'$. We may assume that $M_{ij}\cap H=M_{ij}'\cap H$ (such intersections are exactly the planes of $Y$). Consider the point $$v=(1:0:0:0:0:0),\ v=\bigcap_{i\neq 1}M_{ij}.$$ We may assume that $$\bigcap_{i\neq 1}M_{ij}=\bigcap_{i\neq 1}M'_{ij}$$ (otherwise we can apply an automorphism $\tau'\in B$). From $M_{ij}\cap H=M_{ij}'\cap H$ we deduce $M_{ij}=M'_{ij}$ if $i\neq 1$. Let $F'=0$ be an equation of $Z'$. Then $F'$ is contained in the intersection of ideals $\bigcap_{i\neq 1}(x_{i}, y_{j})$. It is easy to see that such intersection coincides with the ideal $(x_{2}x_{3}, y_{1}y_{2}y_{3})$, so we may assume that $F'=lx_{2}x_{3}-y_{1}y_{2}y_{3}$ for some non-zero linear polynomial $l$. Assume that $H$ has an equation $m=0$. We know that polynomials $F'$ and $x_{1}x_{2}x_{3}-y_{1}y_{2}y_{3}$ coincide on $H$ and $Y$ is irreducible. Thus one can easily see that $m=l-x_{1}$ and $$F'=(x_{1}+m)x_{2}x_{3}-y_{1}y_{2}y_{3}$$ (after normalization). Then the automorphism $\tau$ acting via $$x_{1}\mapsto x_{1}+m,\ x_{2}\mapsto x_{2},\ x_{3}\mapsto x_{3},\ y_{i}\mapsto y_{i}$$ is contained in $B$ and maps $Z$ to $Z'$. Thus the composition $\tau\circ\widetilde{\sigma}$ is a required automorphism of $Z$ inducing $\sigma$.
\end{proof}

Hence  we obtain that the group $\Aut(X)$ is isomorphic to $\Sym_{3}^{2}\rtimes \Cyc_{2}$ if $A^{3}=1$ and $\Sym_{3}^{2}$ if $A^{3}\neq \pm 1, 0$. But in the last case we have an $\Aut(X)$-invariant plane $x_{1}=x_{2}=x_{3}=0$, so the group $\Aut(X)$ is of fiber type by Lemma~\ref{le1}. Hence we may assume that $X$ is given by the equations $$\left\{ x_{0}x_{1}x_{2}-x_{3}x_{4}x_{5}=\sum\limits_{i=0}^{5}x_{i}=0\right\}\subset\mathbb{P}^{5}=\mathbb{P}(V),$$
where $V$ is a 6-dimensional representation of the group $\Sym_{3}^{2}\rtimes \Cyc_{2}$. Let $$V'=\{\sum\limits_{i=0}^{5}x_{i}=0\}\subset V$$ be a 5-dimensional subrepresentation.

Now we are going to describe all the possible minimal subgroups of the group $\Sym_{3}^{2}\rtimes \Cyc_{2}$. We will consider $\Sym_{3}^{2}\rtimes \Cyc_{2}$ as a subgroup of $\Sym_{6}$ with the natural action on the space $V$ with coordinates $x_{i}$.

\begin{pproposition}\label{ut7}The minimal subgroup $G$ satisfying assumption~\ref{as} coincides with one of the following groups:
$$\Sym_{3}^{2},\ \Cyc_{3}^{2}\rtimes \Cyc_{4},\ \Sym_{3}^{2}\rtimes \Cyc_{2},$$ where the first group acts transitively on the set of coordinates $x_{i}$. Conversely, all such subgroups of $\Aut(X)$ are minimal.
\end{pproposition}
\begin{proof}  Recall that there is a natural bijection $l_{ij}\leftrightarrow M_{ij}$ between singular lines and $3$-spaces on $Z$. This bijection induces an $\Aut(X)$-equivariant bijection between singular points and planes on $X$. Since the order of the group $G$ is not divisible by 5, the $G$-orbit of a singular point cannot be of length 5. Hence the group $G$ acts transitively on $\operatorname{Sing}(X)$ by Corollary~\ref{sl1} and all planes form a unique $G$-orbit. Since the group $\operatorname{Cl}(X)$ is a free abelian group generated by the classes of planes (see~\cite{2}), all such subgroups are minimal.  A subgroup of $\Aut(X)$ acts transitively on $\operatorname{Sing}(X)$ if and only if it contains the unique subgroup~$\Cyc_{3}^{2}\subset \Aut(X)$.

On the other hand, $G$ cannot act on $x_{i}$'s and $y_{i}$'s separately since there is no $G$-invariant planes. Also, the group $\Sym_{3}\times \Cyc_{3}$ has only 1- and 2-dimensional irreducible representations, so all subgroups of $G$ isomorphic to $\Sym_{3}\times \Cyc_{3}$ have a 3-dimensional subrepresentation in our 5-dimensional representation $V'$. Thus all such groups have an invariant plane, so they are of fiber type. Thus $G$ is one of the groups from the statement.
\end{proof}

\section{Cubic threefolds with 6 or 5 nodes}

In this section in the following propositions we exclude three remaining cases: J11, J9 and the case of five singular points.
\begin{pproposition} A $G$-variety $X$ of type J11 satisfying the assumption~\ref{as} does not exist.
\end{pproposition}
\begin{proof} Assume that such a $G$-variety $X$ exists. Any variety of type J11 has 6 nodes and no other singularities. Also $X$ contains three planes $\Pi_{1},\ \Pi_{2}$ and $\Pi_{3}$, each plane contains exactly 4 singularities and every singularity is contained in two planes. Let us denote the singularities of $X$ by $p_{1}, ..., p_{6}$. We may assume that the planes $\Pi_{1},\ \Pi_{2}$ and $\Pi_{3}$ contain $(p_{1}, p_{2}, p_{3}, p_{4})$, $(p_{1}, p_{2}, p_{5}, p_{6})$ and $(p_{5}, p_{6}, p_{3}, p_{4})$, respectively. These three planes form a hyperplane section of $X$ and have one common non-singular point. We denote this point by $p_{0}$. Let $l_{12}$ be a line passing through $p_{1}$ and $p_{2}$, let $l_{34}$ be a line passing through $p_{3}$ and $p_{4}$ and let $l_{56}$ be a line passing through $p_{5}$ and $p_{6}$. Let us consider the subgroup $\operatorname{Stab}_{l_{12}}\subset \Aut(X)$. Since $p_{0}$ is an $\Aut(X)$-invariant point and $\{p_{1}, p_{2}\}$ is a $\operatorname{Stab}_{l_{12}}$-invariant set, there is a $\operatorname{Stab}_{l_{12}}$-invariant point $p_{12}$. In the same way we define points $p_{34}$ and $p_{56}$. Obviously, $\{p_{12}, p_{34}, p_{56}\}$ form an $\Aut(X)$-orbit and they are not colinear. So the plane passing through them is $\Aut(X)$-invariant plane that contradicts the Lemma~\ref{le2}.
\end{proof}

\begin{pproposition}\label{pr1} A $G$-variety $X$ of type J9 satisfying the assumption~\ref{as} does not exist.
\end{pproposition}
\begin{proof} Note that for any $7$ points of $\mathbb{P}^{4}$ in general position there exist unique twisted quartic curve passing through them. Let us consider a twisted quartic $C$ passing through singular points of $X$ and one more general point of $X$. If $C$ is not contained in $X$ then $12=C\cdot X\geq 2\cdot6+1=13$. This contradiction shows that the pencil of twisted quartics passing through singularities of $X$ gives us a fibration by rational curves.
\end{proof}

Finally, let $X$ be a cubic hypersurface in $\mathbb{P}^{4}$ with exactly five singular points in general position and $G$ is a subgroup of~$\Aut(X)$ which acts transitively on the singular points of $X$.
\begin{pproposition} In our notation the variety $X$ is $G$-birationally equivalent to a quadric.
\end{pproposition}
\begin{proof} We assume that singular points are $p_{i}=\{x_{j}=\delta_{i}^{j}\}$. The equation of $X$ in such a coordinate system is
$$\sum\limits_{0\leq i<j<k\leq 4} a_{ijk}x_{i}x_{j}x_{k}.$$ The action of $\Aut(X)$ on the set of singularities of $X$ induces a homomorphism $\Aut(X)\to~\Sym_{5}$ with transitive image, so the image contains an element of order 5. Without loss of generality we may assume that the image contains the element $(1, 2, 3, 4, 5)$. Obviously, every element of $\operatorname{PGL}_{5}(\operatorname{k})$ which permutes cyclically singularities of $X$ is the composition of some a diagonal map and the cyclic permutation of coordinates. All such elements are conjugate to each other by diagonal map. This can be seen from the explicit equations for elements of corresponding diagonal matrix. Thus we may assume that $\Aut(X)$ contains a cyclic permutation of the coordinates and $$a_{012}=\xi a_{123}=\xi^{2}a_{234}=\xi^{3}a_{034}=\xi^{4}a_{014},\ a_{013}=\xi a_{124}=\xi^{2}a_{023}=\xi^{3}a_{134}=\xi^{4}a_{024},$$ where $\xi$ is a root of unity of degree 5. After the diagonal change of coordinates $$x_{0}\mapsto x_{0}, x_{1}\mapsto \xi^{2}x_{1}, x_{2}\mapsto \xi^{4}x_{2}, x_{3}\mapsto \xi x_{3}, x_{4}\mapsto \xi^{3}x_{4}$$ we obtain $$a_{012}=a_{123}=a_{234}=a_{034}=a_{014},\ a_{013}=a_{124}=a_{023}=a_{134}=a_{024}.$$Note that all coefficients are nonzero, otherwise the singularities of $X$ are not isolated. Thus we may assume that $a_{012}=1$ and $a_{013}=a$ for some nonzero $a$. Obviously, the homomorphism $\Aut(X)\to\Sym^{5}$ is an injection and every element of $\Aut(X)$ is a composition of a permutation of the coordinates and a diagonal map.

Let us consider the birational automorphism of $\mathbb{P}^{4}$ given by the formula $$\sigma:(x_{0}:x_{1}:x_{2}:x_{3}:x_{4})\mapsto \left(\frac{1}{x_{0}}:\frac{1}{x_{1}}:\frac{1}{x_{2}}:\frac{1}{x_{3}}:\frac{1}{x_{4}}\right).$$ The restriction of $\sigma$ to $X$ gives us a birational $G$-equivariant map from $X$ to a quadric hypersurface $Q\subset\mathbb{P}^{4}$ given by the equation $$x_{1}x_{2}+x_{2}x_{3}+x_{3}x_{4}+x_{4}x_{5}+x_{5}x_{1}+a(x_{1}x_{3}+x_{2}x_{4}+x_{3}x_{5}+x_{4}x_{1}+x_{5}x_{2})=0.$$
One can easily see that either such a quadric is smooth, either it is a cone over smooth quadric surface. In both cases $Q$ is a $G\mathbb{Q}$-Fano threefold and $X$ is not $G$-birationally rigid.
\end{proof}

\end{document}